\documentclass[12pt]{amsart}
\usepackage{amssymb}
\usepackage{amsmath, amscd}
\usepackage{amsthm}
\usepackage{comment}
\usepackage[table,xcdraw]{xcolor}
\usepackage{ulem}
\usepackage{tikz}
\usepackage{float}
\usepackage{graphicx}
\usepackage{circuitikz}
\usepackage[colorlinks=true, linkcolor=blue,urlcolor=blue]{hyperref}

\newtheorem{theorem}{Theorem}[section]

\newtheorem{proposition}[theorem]{Proposition}
\newtheorem{lemma}[theorem]{Lemma}

\newtheorem{corollary}[theorem]{Corollary}
\theoremstyle{definition}

\newtheorem{example}[theorem]{Example}
\newtheorem{definition}[theorem]{Definition}

\newcommand{\bigzero}{\mbox{\normalfont\Large\bfseries 0}}

\begin{document}
	
\author[M. Doostalizadeh]{Mina Doostalizadeh}
\address{Department of Mathematics, Tarbiat Modares University, 14115-111 Tehran Jalal AleAhmad Nasr, Iran}
\email{d\_mina@modares.ac.ir;  m.doostalizadeh@gmail.com}	
\author[A. Moussavi]{Ahmad Moussavi}
\address{Department of Mathematics, Tarbiat Modares University, 14115-111 Tehran Jalal AleAhmad Nasr, Iran}
\email{moussavi.a@modares.ac.ir; moussavi.a@gmail.com}
\author[P. Danchev]{Peter Danchev}
\address{Institute of Mathematics and Informatics, Bulgarian Academy of Sciences, 1113 Sofia, Bulgaria}
\email{danchev@math.bas.bg; pvdanchev@yahoo.com}

\title[Rings whose non-units are square-nil clean]{Rings whose non-units are square-nil clean}
\keywords{unit, (square-)nil clean element (ring), (strongly) square-nil clean element (ring), Matrix ring, Group ring}
\subjclass[2010]{16S34,16U60}

\maketitle




\begin{abstract} We consider in-depth and characterize in certain aspects the class of so-called {\it strongly NUS-nil clean rings}, that are those rings whose non-units are {\it square nil-clean} in the sense that they are a sum of a  nilpotent and a square-idempotent that commutes with each other. This class of rings lies properly between the classes of strongly nil-clean rings and strongly clean rings. In fact, it is proved the valuable criterion that a ring $R$ is strongly NUS-nil clean if, and only if, $a^4-a^2\in Nil(R)$ for every $a\not\in U(R)$. In particular, a ring $R$ with only trivial idempotents is strongly NUS-nil clean if, and only if, $R$ is a local ring with nil Jacobson radical. Some special matrix constructions and group ring extensions will provide us with new sources of examples of NUS-nil clean rings.
\end{abstract}

\section{Introduction and Basic Facts}

Throughout (the rest of) the present paper, all rings are supposed to be associative with identity and all modules are unitary. We denote by $U(R)$, $Nil(R)$ and $Id(R)$ the set of invertible elements, the set of nilpotent elements and the set of idempotent elements in $R$, respectively. Likewise, $J(R)$ denotes the Jacobson radical of $R$, and $Z(R)$ denotes the center of $R$. The ring of $n\times n$ matrices over $R$ and the ring of $n\times n$ upper triangular matrices over $R$ are, respectively, denoted by ${\rm M}_{n}(R)$ and ${\rm T}_{n}(R)$.

\medskip

Mimicking Nicholson \cite{16} and \cite{17}, a ring is called {\it clean} if every its element can be written as a sum of a unit element and an idempotent element. Moreover, a ring $R$ is termed {\it exchange} if, for each $a \in R$, there exists an idempotent $ e \in aR$ such that $1-e \in (1-a)R$. Every clean ring is known to be exchange, but the converse implication is manifestly {\it not} valid in general; however it is true in the abelian case (see, e.g., \cite[Proposition 1.8]{3}).

In \cite{17}, an element $c$ of a ring $R$ is defined to be {\it strongly clean} if there is an idempotent $e\in R$, which commutes with $c$, such that $c-e$ is invertible in $R$. Analogously, a strongly clean ring is the one in which every element is strongly clean. Furthermore, it is shown that a strongly clean element of a ring satisfies a generalized version of the classical Fitting's Lemma. This yields (cf. \cite[Section 10]{Boro}) that each strongly $\pi$-regular element of a ring is strongly clean. (We also refer to \cite{ckl} for an excellent overview of the relationship between the clean property and other classical ring theory notions.)

\medskip

Next, imitating Diesl \cite{2}, a ring $R$ is said to be {\it (strongly) nil-clean} if, for each $r\in R$, there are $q\in Nil(R)$ and $e\in Id(R)$ such that $r=q+e$ (in addition, $qe=eq$). It was proven there many fundamental properties as well as developed a general theory of (strongly) nil-clean rings.

\medskip

Following Danchev et al. \cite{5}, a ring $R$ is called {\it generalized strongly nil-clean}, and briefly abbreviated by GSNC, if any non-invertible element of $R$ is strongly nil-clean.

\medskip

Now, we say that an idempotent $a\in R$ is said to be {\it square-idempotent}, provided $a^2=a^4$. Motivated by the above discussion, we introduce and study the notion of {\it NUS-nil clean} rings. So, a ring $R$ is called {\it (strongly) square-nil clean} if, for each $a\in R$, there are a square-idempotent $e\in R$ and an element $n\in Nil(R)$ such that $a=e+n$ (in addition, $en=ne$). If any non-unit (i.e., any non-invertible element) in a ring $R$ is    (strongly) square-nil clean, we say $R$ is (strongly) NUS-nil clean. The class of strongly NUS-nil clean rings lies strictly between the strongly nil-clean rings and strongly clean rings. Specifically, we establish that: a ring $R$ is strongly NUS-nil clean if, and only if, $a^4-a^2\in Nil(R)$ for all $a\not\in U(R)$. We, moreover, prove that this property passes to corner rings and is {\it not} Morita invariant. Also, the direct product of strongly NUS-nil clean rings is too a strongly NUS-nil clean. In particular, strongly NUS-nil clean rings are always strongly $\pi$-regular. Likewise, the Jacobson radical of a strongly NUS-nil clean ring is a nil-ideal. Finally, a ring $R$ with only trivial idempotents is strongly NUS-nil clean if, and only if, $R$ is a local ring with nil Jacobson radical. Certain matrix constructions and group ring extensions will reach us with some new types of examples of NUS-nil clean rings.

\section{Strongly NUS-Nil Clean Rings}

In this section, we officially introduce the concept of strongly NUS-nil clean rings and investigate their elementary but useful properties.

\begin{definition}\label{2.1} Let $R$ be a ring. An element $e\in R$ is said to be {\it square-idempotent}, provided $e^2=e^4$. Additionally, an element $a\in R$ is said to be {\it square-nil clean} if $a=e+n$, where $e$ is a square-idempotent and $n$ is a nilpotent. In particular, if $en=ne$, the element is called {\it strongly square-nil clean}. Particularly, a ring $R$ is called {\it (strongly) NUS-nil clean} if every non-unit (i.e., every non-invertible) element in $R$ is (strongly) square-nil clean.
\end{definition}

We begin our work with the next series of preliminary technicalities.

\begin{lemma}\label{2.2}
Every strongly NUS-nil clean element is clean.
\end{lemma}

\begin{proof}
Assume $a = e + n$, where $a$ is a non-unit, $e$ is a square-idempotent and $n$ is a nilpotent element of a ring $R$ that commutes with $e$. It is now easily seen that $a=1-e^2+e-1+e^2+n$ and $e-1+e^2+n\in U(R)$, as needed.
\end{proof}

Our immediate consequence is the following.

\begin{corollary}\label{2.3}
Every strongly NUS-nil clean ring is strongly clean.
\end{corollary}

In other words, the class of strongly NUS-nil clean rings lies in a proper way between the classes of strongly nil-clean rings and strongly clean rings. Indeed, it is worthwhile noting that $\mathbb{Z}_3\times \mathbb{Z}_3$, $M_2(\mathbb{Z}_3)$ and $M_2(\mathbb{Z}_2)$ are all strongly NUS-nil clean rings that are {\it not} strongly   nil-clean, while $\mathbb{Z}_{2}[[x]]$ and $\mathbb{Z}_{3}[[x]]$ are strongly clean rings that are {\it not} strongly NUS-nil clean.

\begin{lemma}\label{2.4}
Let $R_{i}$ be a ring for all $i\in I$. The direct product $\prod_{i=1}^{n} R_i$ is strongly NUS-nil clean if, and only if, each direct component $R_{i}$ is strongly square-nil clean.
\end{lemma}

\begin{proof}
Letting $a_{i}\in R_{i}$ for some index $i$, where $a_{i}\notin U(R_{i})$, whence $$(0,\ldots,0,a_{i},0,\ldots,0)\notin  U(\prod^{n}_{i=1}R_{i}).$$ So, $(0,\ldots,0,a_{i},0,\ldots,0)$ is strongly square-nil clean, and hence $a_{i}$ is strongly square-nil clean. That is why, any $R_{i}$ is strongly square-nil clean, as required.
\end{proof}

\begin{example}
Obvious calculations illustrate that the commutative ring $\mathbb{Z}_5$ is strongly NUS-nil clean, but is not strongly square-nil clean.	
\end{example}

Recall the pivotal fact that, in view of \cite[Lemma 2.6]{zhou}, if $2 \in U(R)$ and $a^3 -a$ is a nilpotent, then there exists a polynomial $\alpha(t) \in Z[t]$ such that $\alpha(a)^3 = \alpha(a)$ and $a -\alpha(a)$ is a nilpotent.

\medskip

We now manage to proceed by proving the following helpful necessary and sufficient condition.

\begin{theorem}\label{7} Let $R$ be a ring. Then, $R$ is strongly NUS-nil clean if, and only if, $a^4-a^2\in Nil(R)$ for every $a\not\in U(R)$.
\end{theorem}

\begin{proof} Assume that $R$ is strongly NUS-nil clean. Then, for all $a\not\in U(R)$, there are $e^4=e^2\in R$ and $n\in Nil(R)$ such that $a=e+n$ and $en=ne$. Thus, $a^2=e^2+n(2e+n)$. So, $$a^4=e^2+n(2e+n)(2e^2+n(2e+n)).$$ Therefore, clearly $a^4-a^2\in Nil(R)$, as desired.

Conversely, assume that $a^4-a^2\in Nil(R)$ for every $a\not\in U(R)$. It follows that $a^3-a\in Nil(R)$ for each $a\not\in U(R)$. If, for a moment, $2\in U(R)$, then the aforementioned \cite[Lemma 2.6]{zhou} applies to get that there exists $\alpha (a)\in \mathbb{Z}$ such that $\alpha (a)^3=\alpha(a)$ and $a-\alpha(a)\in Nil(R)$. So, there exists $m\in Nil(R)$ such that $a=\alpha(a)+m$. It is now easy to see that $\alpha(a)^4=\alpha(a)^2$ and $m\alpha(a)=\alpha(a)$, so that we are done.\\

Next, if $2\not\in U(R)$, then $2=f+q$, where $f^2=f^4$ and $\in Nil(R)$ with $qf=fq$. It thus follows that $$2^4-2^2=q(2f+q)(f^2+(2f+q)q-1)\in Nil(R).$$ So, $12=3\times 2\times 2\in Nil(R)$ which insures that $6\in Nil(R)$. On the other hand, since, for every $a\not\in Nil(R)$, $a^3-a\in Nil(R)$, we get $(a^3-a)^6\in Nil(R)$ which forces that $a^{18}-a^6\in Nil(R)$. Thus, $a^{13}-a \in Nil(R)$. From this and $a^3-a\in Nil(R)$, we derive $a^{13}-a^3\in Nil(R)$ and hence $a^{11}-a\in Nil(R)$. From this and $a^3-a\in Nil(R)$, we deduce $a^{11}-a^3\in Nil(R)$ whence $a^9-a\in Nil(R)$. As $a^3-a\in Nil(R)$, ut must be that $a^9-a^3\in Nil(R)$. It, therefore, follows that $a^7-a\in Nil(R)$ and so $a(a^6-1)\in Nil(R)$. But, since $6\in Nil(R)$, we obtain that $a(a-1)^6\in Nil(R)$ enabling us that $a(a-1)\in Nil(R)$. Furthermore, again owing to \cite[Lemma 2.6]{zhou}, there exists $\beta(a)\in \mathbb{Z}[a]$ such that $\beta (a)^2=\beta(a)$ and $a-\beta(a)\in Nil(R)$. Finally, it follows that $a=\beta(a)+k$, where $k\in Nil(R)$ and $k\beta(a)=\beta(a)k$, as wanted.
\end{proof}

A ring $R$ is known to be {\it strongly $\pi$-regular}, provided that, for any $a \in R$, there exists $n \in \mathbb{N}$ such that $a^n \in a^{n+1}R$.

\medskip

We now come to our critical statement.

\begin{proposition}\label{2.12}
Every strongly NUS-nil clean ring is strongly $\pi$-regular.
\end{proposition}

\begin{proof}
Let $R$ be a strongly NUS-nil clean ring and $a \in R$. If $a \in U(R)$, then $a$ is obviously a strongly $\pi$-regular element. If, however, $a \notin U(R)$, then, in virtue of Theorem \ref{7}, there exists $m \in \mathbb{N}$ such that $(a^4-a^2)^m = 0$. It, thereby, follows at once that $a^{2m} = a^{2m+1}r$ for some $r \in R$. Consequently, $R$ is a strongly $\pi$-regular ring, as expected.
\end{proof}	

We next continue with further crucial claims.

\begin{lemma}\label{2.14}
Let $R$ be a ring and $0\neq e=e^2\in R$. If $R$ is strongly NUS-nil clean, then so is the corner subring $eRe$.
\end{lemma}

\begin{proof}
Let $a \in eRe$ such that $a\not\in U(eRe)$. We write $a = ea = ae = eae$. If $a \in U(R)$, then there exists $b \in R$ such that $ab = ba = 1$, which assures $a(ebe) = (ebe)a = e$, a contradiction. Therefore, $a \notin U(R)$. Hence, according to Proposition \ref{7}, we have $$a^4- a^2 \in {\rm Nil}(R) \cap eRe \subseteq {\rm Nil}(eRe),$$ as promised.
\end{proof}

\begin{lemma}\label{8}
Let $R$ be a strongly NUS-nil clean ring. Then, $J(R)$ is a nil-ideal of $R$.
\end{lemma}

\begin{proof} Let $j\in J(R)$. Thus, $j\not\in U(R)$. Now, using Lemma \ref{7}, we detect $j^4-j^2\in Nil(R)$. It apparently follows that $j^2(j^2-1)\in Nil(R)$. So, $j^{2t}(j^2-1)^t=0$ for some positive integer $t$. However, as $j^2-1\in U(R)$, we get $j^{2t}=0$ and so $j\in Nil(R)$, as pursued.
\end{proof}

\begin{proposition}\label{2.13}
Let $R$ be a ring. The following hold:	
\par
(i) For any nil-ideal $I \subseteq R$, $R$ is strongly NUS-nil clean if, and only if, $R/I$ is strongly NUS-nil clean.
\par
(ii) A ring $R$ is strongly NUS-nil clean if, and only if, $J(R)$ is nil and $R/J(R)$ is strongly NUS-nil clean.
\end{proposition}

\begin{proof}
(i) Suppose $R$ is a strongly NUS-nil clean ring and put $\overline{R} := R/I$. If $\bar{a} \notin U(\overline{R})$, then $a \notin U(R)$, which ensures with the aid of Theorem \ref{7} that $a^4 - a^2 \in {\rm Nil}(R)$, so that $\bar{a}^4- \bar{a}^2 \in {\rm Nil}(\overline{R})$.

Conversely, suppose $\overline{R}$ is a strongly NUS-nil clean ring. If $a \notin U(R)$, then $\bar{a} \notin U(\overline{R})$, and thus Theorem \ref{7} guarantees that $\bar{a}^4- \bar{a}^2 \in {\rm Nil}(\overline{R})$. Therefore, there exists $k \in \mathbb{N}$ such that $(a^4- a^2)^k \in I \subseteq {\rm Nil}(R)$.

(ii) Utilizing Lemma \ref{8} and part (i) of the proof, the arguments are complete.
\end{proof}

\begin{corollary}\label{10}
Let $I$ be an ideal of a ring $R$. Then, the following are equivalent:
\begin{enumerate}
\item
$R/I$ is strongly NUS-nil clean.
\item
$R/I^n$	is strongly NUS-nil clean for all $n \in \mathbb{N}$.
\item
$R/I^n$ is strongly NUS-nil clean for some $n \in \mathbb{N}$.
\end{enumerate}	
\end{corollary}

\begin{proof}
(i) $\Longrightarrow$ (ii). For any $n \in \mathbb{N}$, we have $\dfrac{R/I^n}{I/I^n} \cong R/I$. Since $I/I^n$ is a nil-ideal of $R/I^n$ and $R/I$ is strongly NUS-nil clean, then Proposition \ref {2.13} gives that $R/I^n$ is a strongly NUS-nil clean ring.\\
(ii) $\Longrightarrow$ (iii). This implication is trivial.\\
(iii) $\Longrightarrow$ (i). For any ideal $I$ of $R$, we have $\dfrac{R/I^n}{I/I^n} \cong R/I$, so that, via Proposition \ref{2.13}, we can conclude that $R/I$ is strongly NUS-nil clean, as stated.
\end{proof}

\begin{proposition}\label{2.9}
Let $R$ be a ring. Then, the following are equivalent:
\begin{enumerate}
\item
$R$ is strongly square nil clean.
\item
${\rm T}_{n}(R)$ is strongly NUS-nil clean for all $n \in \mathbb{N}$.
\item
${\rm T}_n(R)$ is strongly NUS-nil clean for some $n \geq 2$.
\end{enumerate}
\end{proposition}

\begin{proof}
(ii) $\Rightarrow$ (iii). This is trivial.\\
(iii) $\Rightarrow$ (i). Let $a\in R$. Then, it must be that
$$A=\begin{pmatrix}
	a & 0 \\
	0& 0 \\
\end{pmatrix}\in {\rm T}_2(R).$$ It is evident that $A$ is non-invertible in ${\rm T}_2(R)$. By hypothesis, we can find a square idempotent
$E=\begin{pmatrix}
	e_{11} &e_{12} \\
	0& e_{22}  \\
\end{pmatrix}$ and a nilpotent
$Q=\begin{pmatrix}
	q_{11} &q_{12} \\
	0& q_{22}  \\
\end{pmatrix}$ such that
$A=E+Q$ and $EQ=QE$. It now follows by plain inspection that $a=e_{11}+q_{11}$ and $e_{11}q_{11}=q_{11}e_{11}$. Hence, $e_{11}^4=e_{11}^2$, and $q_{11}$ is a nilpotent in $R$. Thus, $a$ has strongly NUS-nil clean decomposition. Therefore, point (i) is valid.\\
(i) $\Rightarrow$ (ii). It suffices with Lemma \ref{10} at hand to prove only that the quotient $T_n(R)/J(T_n(R))$ is strongly NUS-nil clean. To this target, observe that $$T_n(R)/J(T_n(R))\cong \sqcap_{i=1}^nR/J(R).$$ So, we need just to show that $\sqcap_{i=1}^nR/J(R)$ is strongly NUS-nil clean. In fact, Lemma \ref{2.4} is a guarantor that $\sqcap_{i=1}^nR/J(R)$ is strongly NUS-nil clean if, and only if, $R/J(R)$ is strongly square-nil clean. It is now easy to see that, if $R$ is strongly square-nil clean, then $R/J(R)$ is strongly-square nil clean, as asked for.
\end{proof}

Let $R$ be a ring and $M$ a bi-module over $R$. The {\it trivial extension} of $R$ and $M$ is stated as
\[ T(R, M) = \{(r, m) : r \in R \text{ and } m \in M\}, \]
with addition defined component-wise and multiplication defined by
\[ (r, m)(s, n) = (rs, rn + ms). \]

\medskip

The next two consequences arrived quite naturally.

\begin{corollary}\label{2.17}
Let $R$ be a ring and $M$ a bi-module over $R$. Then, the following statements are equivalent:
\begin{enumerate}
\item
$T(R, M)$ is a strongly NUS-nil clean ring.
\item
$R$ is a strongly NUS-nil clean ring.
\end{enumerate}
\end{corollary}

\begin{proof}
Set $A:={\rm T}(R, M)$ and consider $I:={\rm T}(0, M)$. It is not too hard to verify that $I$ is a nil-ideal of $A$ such that $\dfrac{A}{I} \cong R$. So, the result follows directly from Proposition \ref{2.13}.
\end{proof}

Let $\alpha$ be an endomorphism of $R$ and $n$ a positive integer. Consider the {\it skew triangular matrix ring}
$${\rm T}_{n}(R,\alpha )=\left\{ \left. \begin{pmatrix}
	a_{0} & a_{1} & a_{2} & \cdots & a_{n-1} \\
	0 & a_{0} & a_{1} & \cdots & a_{n-2} \\
	0 & 0 & a_{0} & \cdots & a_{n-3} \\
	\ddots & \ddots & \ddots & \vdots & \ddots \\
	0 & 0 & 0 & \cdots & a_{0}
\end{pmatrix} \right| a_{i}\in R \right\}$$
with addition point-wise and multiplication given by:
\begin{align*}
&\begin{pmatrix}
		a_{0} & a_{1} & a_{2} & \cdots & a_{n-1} \\
		0 & a_{0} & a_{1} & \cdots & a_{n-2} \\
		0 & 0 & a_{0} & \cdots & a_{n-3} \\
		\ddots & \ddots & \ddots & \vdots & \ddots \\
		0 & 0 & 0 & \cdots & a_{0}
	\end{pmatrix}\begin{pmatrix}
		b_{0} & b_{1} & b_{2} & \cdots & b_{n-1} \\
		0 & b_{0} & b_{1} & \cdots & b_{n-2} \\
		0 & 0 & b_{0} & \cdots & b_{n-3} \\
		\ddots & \ddots & \ddots & \vdots & \ddots \\
		0 & 0 & 0 & \cdots & b_{0}
	\end{pmatrix}  =\\
	& \begin{pmatrix}
		c_{0} & c_{1} & c_{2} & \cdots & c_{n-1} \\
		0 & c_{0} & c_{1} & \cdots & c_{n-2} \\
		0 & 0 & c_{0} & \cdots & c_{n-3} \\
		\ddots & \ddots & \ddots & \vdots & \ddots \\
		0 & 0 & 0 & \cdots & c_{0}
\end{pmatrix},
\end{align*}
where $$c_{i}=a_{0}\alpha^{0}(b_{i})+a_{1}\alpha^{1}(b_{i-1})+\cdots +a_{i}\alpha^{i}(b_{1}),~~ 1\leq i\leq n-1
.$$

\medskip

We, hereafter, denote the elements of ${\rm T}_{n}(R, \alpha)$ by $(a_{0},a_{1},\ldots , a_{n-1})$. If $\alpha $ is the identity endomorphism, then ${\rm T}_{n}(R,\alpha )$ is a subring of the upper triangular matrix ring ${\rm T}_{n}(R)$.

\begin{corollary}\label{2.20}
Let $R$ be a ring. Then, the following statements are equivalent:
\begin{enumerate}
\item
${\rm T}_{n}(R,\alpha)$ is a strongly NUS-nil clean ring.
\item
$R$ is a strongly NUS-nil clean ring.
\end{enumerate}
\end{corollary}

\begin{proof}
Choose
$$I:=\left\{
\left.
\begin{pmatrix}
		0 & a_{12} & \ldots & a_{1n} \\
		0 & 0 & \ldots & a_{2n} \\
		\vdots & \vdots & \ddots & \vdots \\
		0 & 0 & \ldots & 0
	\end{pmatrix} \right| a_{ij}\in R \quad (i\leq j )
	\right\}.$$
Then, one easily checks that $I^{n}=0$ and $\dfrac{{\rm T}_{n}(R,\alpha )}{I} \cong R$. Consequently, Proposition \ref{2.13} employs to get the desired result.
\end{proof}

Furthermore, one mentions that Wang introduced in \cite{13} the matrix ring ${\rm S}_{n,m}(R)$ as follows: supposing $R$ is a ring, consider the matrix ring ${\rm S}_{n,m}(R)=$

$$\left\{ \begin{pmatrix}
   a & b_1 & \cdots & b_{n-1} & c_{1n} & \cdots & c_{1 n+m-1}\\
   \vdots  & \ddots & \ddots & \vdots & \vdots & \ddots & \vdots \\
   0 & \cdots & a & b_1 & c_{n-1,n} & \cdots & c_{n-1,n+m-1} \\
   0 & \cdots & 0 & a & d_1 & \cdots & d_{m-1} \\
   \vdots  & \ddots & \ddots & \vdots & \vdots & \ddots & \vdots \\
   0 & \cdots & 0 & 0  & \cdots & a & d_1 \\
   0 & \cdots & 0 & 0  & \cdots & 0 & a
\end{pmatrix}\in {\rm T}_{n+m-1}(R) : a, b_i, d_j,c_{i,j} \in R \right\}.$$

\noindent Likewise, let ${\rm T}_{n,m}(R)$ be

$$\left\{ \left(\begin{array}{@{}c|c@{}}
  \begin{matrix}
  a & b_1 & b_2 & \cdots & b_{n-1} \\
  0 & a & b_1 & \cdots & b_{n-2} \\
  0 & 0 & a & \cdots & b_{n-3} \\
  \vdots & \vdots & \vdots & \ddots & \vdots \\
  0 & 0 & 0 & \cdots & a
  \end{matrix}
  & \bigzero \\
\hline
  \bigzero &
  \begin{matrix}
  a & c_1 & c_2 & \cdots & c_{m-1} \\
  0 & a & c_1 & \cdots & c_{m-2} \\
  0 & 0 & a & \cdots & c_{m-3} \\
  \vdots & \vdots & \vdots & \ddots & \vdots \\
  0 & 0 & 0 & \cdots & a
  \end{matrix}
\end{array}\right)\in {\rm T}_{n+m}(R) : a, b_i,c_j \in R \right\},$$

\noindent and let we state

$${\rm U}_{n}(R)=\left\{ \begin{pmatrix}
   a & b_1 & b_2 & b_3 & b_4 & \cdots & b_{n-1} \\
   0 & a & c_1 & c_2 & c_3 & \cdots & c_{n-2} \\
   0 & 0 & a & b_1 & b_2 & \cdots & b_{n-3} \\
   0 & 0 & 0 & a & c_1 & \cdots & c_{n-4} \\
   \vdots & \vdots & \vdots & \vdots &  &  & \vdots \\
   0 &0 & 0 & 0 & 0 & \cdots & a
\end{pmatrix}\in {\rm T}_{n}(R) :  a, b_i, c_j \in R \right\}.$$

Thereby, we have the following.

\begin{example}\label{exa3.29}
Let $R$ be a ring. Then, the following statements are equivalent:
\par
(i) ${\rm S}_{n,m}(R)$ is a NUS-nil clean ring.
\par
(ii) ${\rm T}_{n,m}(R)$ is a NUS-nil clean ring.
\par
(iii) ${\rm U}_{n}(R)$ is a NUS-nil clean ring.
\par
(iv) $R$ is a NUS-nil clean ring.
\end{example}

Let $\alpha$ be an endomorphism of $R$. We denote by $R[x,\alpha ]$ the {\it skew polynomial ring} whose elements are the polynomials over $R$; the addition is defined as usual, and the multiplication is defined by the equality $xr=\alpha (r)x$ for any $r\in R$. So, there is a ring isomorphism $$\varphi : \dfrac{R[x,\alpha]}{\langle x^{n}\rangle }\rightarrow {\rm T}_{n}(R,\alpha),$$ given by $$\varphi (a_{0}+a_{1}x+\ldots +a_{n-1}x^{n-1}+\langle x^{n} \rangle )=(a_{0},a_{1},\ldots ,a_{n-1})$$ with $a_{i}\in R$, $0\leq i\leq n-1$. Thus, one deduces that ${\rm T}_{n}(R,\alpha )\cong \dfrac{R[x,\alpha ]}{\langle  x^{n}\rangle}$, where $\langle x^{n}\rangle$ is the ideal generated by $x^{n}$.

Besides, $R[[x, \alpha]]$ denotes the {\it ring of skew formal power series} over $R$; that is, all formal power series in $x$ with coefficients from $R$ with multiplication defined by $xr = \alpha(r)x$ for all $r \in R$. On the other hand, we know that the isomorphism $\dfrac{R[x,\alpha ]}{\langle x^{n}\rangle}\cong \dfrac{R[[x,\alpha ]]}{\langle x^{n}\rangle}$ is fulfilled.

\medskip

We, thus, extract the following three consequences.

\begin{corollary}\label{2.21}
Let $R$ be a ring with an endomorphism $\alpha$ such that $\alpha (1)=1$. Then, the following statements are equivalent:
\begin{enumerate}
\item
$\dfrac{R[x,\alpha ]}{\langle x^{n}\rangle }$ is a strongly NUS-nil clean ring.
\item
$\dfrac{R[[x,\alpha ]]}{\langle x^{n}\rangle }$ is a strongly NUS-nil clean ring.
\item
$R$ is a strongly NUS-nil clean ring.
\end{enumerate}
\end{corollary}

\begin{corollary}
Let $R$ be a ring. Then, the following statements are equivalent:
\begin{enumerate}
\item
$\dfrac{R[x]}{\langle x^{n}\rangle }$ is a strongly NUS-nil clean ring.
\item
$\dfrac{R[[x]]}{\langle x^{n}\rangle }$ is a strongly NUS-nil clean ring.
\item
$R$ is a strongly NUS-nil clean ring.
\end{enumerate}
\end{corollary}

\begin{corollary}\label{2.57}
Let $R$ be a ring, and let
\begin{center}
$S_{n}(R):=\left\lbrace (a_{ij})\in T_{n}(R)\, | \, a_{11}=a_{22}=\cdots=a_{nn}\right\rbrace.$
\end{center}
Then, the following statements are equivalent:
\begin{enumerate}
\item	
$S_{n}(R)$ is a strongly NUS-nil clean ring.
\item	
$R$ is a strongly NUS-nil clean ring.
\end{enumerate}
\end{corollary}

\begin{proof}	
Assuming $I=\{(a_{ij}) \in S_n(R) : a_{11}=0\}$, it is rather evident that $I$ is a nil-ideal of $S_n(R)$ such that $S_n(R)/I \cong R$ holds, as inspected.
\end{proof}

The following example demonstrates that the matrix constructions are more complicated than we anticipate. Specifically, we exhibit a concrete construction of a strongly NUS-nil clean ring which is {\it not} strongly square-nil clean.

\begin{example}\label{2.24}
Let $R={\rm M}_2(\mathbb{Z}_2)$. It can be shown by a direct computations that $$R = U(R) \cup {\rm Id}(R) \cup {\rm Nil}(R).$$ Thus, $R$ is a strongly NUS-nil clean ring, but it is {\it not} strongly square-nil clean. Indeed, by contrary, assume that $T_2(\mathbb{Z}_2)$ is strongly square nil clean. Then, $A=\begin{pmatrix}
		1 & 1  \\
		1 & 0 \\
	\end{pmatrix} \in T_2(\mathbb{Z}_2)$. So, there are $E^4=E^2\in T_2(\mathbb{Z}_2)$ and $N\in Nil(T_2(\mathbb{Z}_2))$ such that $A=E+N$ and $EN=NE$. It follows now that $A^4-A^2=\begin{pmatrix}
		1 & 0 \\
		0 & 1
	\end{pmatrix}\in Nil(T_2(\mathbb{Z}_2))$, that is an obvious contradiction. Therefore, $T_2(\mathbb{Z}_2)$ is really {\it not} strongly NUS-nil clean, as claimed.
\end{example}

The next assertion arises logically.

\begin{proposition}\label{2.25}
For any ring $R \neq 0$ and any integer $n \ge 3$, the ring ${\rm M}_n(R)$ is not strongly NUS-nil clean.
\end{proposition}

\begin{proof}
It suffices to establish that ${\rm M}_3(R)$ is {\it not} a strongly NUS-nil clean ring having in mind Lemma \ref{2.14}. To this goal, consider the matrix
	$$A =\begin{pmatrix}
		1 & 1 & 0 \\
		1 & 0 & 0 \\
		0 & 0 & 0
\end{pmatrix} \notin U({\rm M}_3(R)).$$ Then, we have $$A^4- A^2 \notin {\rm Nil}({\rm M}_3(R)).$$ Consequently, Theorem \ref{7} is applicable to get that $R$ cannot be a strongly NUS-nil clean ring, as asserted.
\end{proof}

An immediate consequence is the following one.

\begin{corollary}\label{2.35}
Let $R$ be a strongly NUS-nil clean ring. Then, for any $n>2$, there does not exist $0\neq e\in {\rm Id}(R)$ such that $eRe\cong {\rm M}_{n}(S)$ for some non-zero ring $S$.
\end{corollary}

\begin{proof}
Assume on the contrary that there exists $0\neq e\in {\rm Id}(R)$ such that $eRe\cong {\rm M}_{n}(S)$ for some non-zero ring $S$. Since $R$ is strongly NUS-nil clean, it follows from Lemma \ref{2.14} that $eRe$ has to be strongly NUS-nil clean too, and so ${\rm M}_{n}(S)$ is also strongly NUS-nil clean, implying a contradiction with Proposition \ref{2.25}, as expected.
\end{proof}

The following affirmation is somewhat surprising.

\begin{lemma}\label{2.26}
Let ${\rm M}_2(R)$ be a strongly NUS-nil clean ring. Then, $R$ is a strongly square-nil clean ring.
\end{lemma}

\begin{proof}
Let $a \in R$. Then, one finds that $$A=\begin{pmatrix}
		a & 0 \\ 0 & 0
\end{pmatrix} \notin U({\rm M}_2(R)).$$ Thus, Lemma \ref{7} works to get that $A^4-A^2 \in {\rm Nil}({\rm M}_2(R))$. So, $a^4- a^2 \in {\rm Nil}(R)$ whence $a^3- a \in {\rm Nil}(R)$. Therefore, bearing in mind 
\cite[Theorem 2.12]{zhou}, one writes that $a=e+n$, where $e^3=e$ and $en=ne$. As $e^4=e^2$, $e$ is a square-idempotent and so $R$ is a strongly-square nil clean ring, as formulated.
\end{proof}

We now intend to examine some structural characterizations.

\begin{lemma}\label{2.27}
If $R$ is a local ring with nil $J(R)$, then $R$ is strongly NUS-nil clean.
\end{lemma}

\begin{proof}
Let $a\in R$ and $a\notin U(R)$. Since $R$ is local, $a\in J(R)$ and hence $a\in {\rm Nil}(R)$. So, $a$ is a nilpotent element, and thus it is a strongly NUS-nil clean element, as required.
\end{proof}

We now can extract the following criterion.

\begin{corollary}\label{2.50}
Let $R$ be a ring with only trivial idempotents. Then, $R$ is strongly NUS-nil clean if, and only if, $R$ is a local ring with $J(R)$ nil.
\end{corollary}

\begin{proof}
Assume that $R$ is a strongly NUS-nil clean ring, so $J(R)$ is nil in accordance with Lemma \ref{8}. If $a \notin U(R)$, then we have $a=q+u$, where $qu=uq$, $q\in Nil(R)$ and either $u^2=1$ or $u^2=0$. Since $a$ is not a unit, it must be that $a=u+q$, where $qu=uq$ and $u^2=0$ giving $a\in Nil(R)$. Thus, exploiting \cite[Proposition 19.3]{14}, $R$ must be a local ring.
	
Oppositely, suppose $R$ is a local ring with a nil Jacobson radical $J(R)$. So, for each $a \notin U(R)$, we have $a \in J(R) \subseteq Nil(R)$, whence $a$ is a nil-clean element, as requested.
\end{proof}

\begin{lemma}\label{2.55}
Suppose $R$ is a strongly NUS-nil clean ring ring and $2 \notin U(R)$. Then, either $2 \in {\rm Nil}(R)$ or $6 \in {\rm Nil}(R)$.
\end{lemma}

\begin{proof}
If $2\not\in Nil(R)$, then Theorem \ref{7} can be applied to derive that $2^4-2\in Nil(R)$. Hence, we get $12\in Nil(R)$ and so $6\in Nil(R)$.
\end{proof}

\begin{lemma}\label{2.29}
Let $R$ be a ring and $2 \in J(R)$. Then, the following conditions are equivalent:
\begin{enumerate}
\item
$R$ is a strongly NUS-nil clean ring.
\item
$R$ is a GSNC ring.
\end{enumerate}
\end{lemma}

\begin{proof}
(ii) $\Longrightarrow$ (i). It is straightforward.\\
(i) $\Longrightarrow$ (ii). It is sufficient to prove only that $a^2-a\in Nil(R)$ for each $a\in R$. To that end, choose $a\in R$. Thus, Theorem~\ref{7} allows us to infer that $a^4-a^2\in Nil(R)$, and so $a^2(1-a^2)\in Nil(R)$. But, as $2\in Nil(R)$, we get $a^2(1-a^2-2a+2a^2)\in Nil(R)$ and, therefore, $a^2(1-a)^2\in Nil(R)$. Thus, $a^2-a\in Nil(R)$, as needed.
\end{proof}

\begin{lemma}\label{2.56}
Suppose $R$ is a ring such that $2 \notin U(R)$. Then, the following items are equivalent:
\begin{enumerate}
\item
$R$ is a strongly NUS-nil clean ring.
\item
Either $R$ is a GSNC ring or $R$ is a strongly square-nil clean ring.
\end{enumerate}
\end{lemma}

\begin{proof}
(ii) $\Longrightarrow$ (i). This is routine.\\
(i) $\Longrightarrow$ (ii). Applying Lemma \ref{2.55}, we have that either $2 \in \text{Nil}(R)$ or $6 \in \text{Nil}(R)$. If, for a moment, $2 \in \text{Nil}(R)$, then $R$ is a GSNC ring using Lemma \ref{2.29}. However, if $6 \in \text{Nil}(R)$, we may decompose with the help of the Chinese Remainder Theorem $R \cong R_1 \oplus R_2$, where $2\in Nil(R_1)$ and $3\in Nil(R_2)$. Now, Lemma \ref{2.4} tells us that $R_1$ and $R_2$ are both strongly square-nil clean rings. Moreover, since $2 \in \text{Nil}(R_1)$, $R_1$ is even a strongly nil-clean ring. It thus follows that $R$ is a strongly square-nil clean ring, and hence $R$ itself is a strongly NUS-nil clean ring, as promised.
\end{proof}

The following claim is key.

\begin{lemma}\label{2.30}
Let $R$ be a ring. Then, the following issues are equivalent:
\begin{enumerate}
\item
$R$ is a strongly square-nil clean ring.
\item
$R$ is a strongly NUS-nil clean ring and, for every $u\in U(R)$, $u^2=1+n$, where $n\in Nil(R)$.
\end{enumerate}
\end{lemma}

\begin{proof}
(i) $\Rightarrow$ (ii). It is readily to see that $R$ is strongly NUS-nil clean. Let $u\in U(R)$. Thus, $u=e+m$, where $e^2=e^4$ and $m\in Nil(R)$ with $em=me$. It follows further from Theorem~\ref{7} that $u^4-u^2\in Nil(R)$. So, $1-u^2\in Nil(R)$. So, there exists $n\in Nil(R)$ such that $u^2=1+n$, as desired.\\
(ii) $\Rightarrow$ (i). Assume that  $a\in R$. If $a\not\in U(R)$, then $a=f+q$, where $f^2=f^4$ and $q\in Nil(R)$ with $qf=fq$. If $a\in U(R)$, then $a^2=1+n$, where $n\in Nil(R)$. Hence, in either case, the element $a$ has strongly  square-nil clean decomposition, as wanted.
\end{proof}

We are now prepared to prove our principal result that sounds somewhat curious and is a partial converse of Lemma~\ref{2.26}.

\begin{theorem}\label{2.38}
Let $R$ be simultaneously a Noetherian local and strongly square nil clean ring. Then, ${\rm M}_2(R)$ is a strongly NUS-nil clean ring.
\end{theorem}

\begin{proof}
Firstly, we show that $6\in Nil(R)$. If, for a moment, $2\not\in U(R)$, then Lemma \ref{2.55} employs to get either $2\in Nil(R)$ or $6\in Nil(R)$. So, in this case, we have $6\in Nil(R)$, because $2\in Nil(R)$ would imply that $6\in Nil(R)$.\\
If $2\in U(R)$, then by hypothesis there exist $e^2=e^4\in R$ and $n\in Nil(R)$ such that $2=e+n$ and $en=ne$. So, $$2^3-2=e^3-e+n(3e^2+3en+n^2-1)\in Nil(R).$$ Therefore, $6\in Nil(R)$ in either case.\\ Thus, one can decompose with the aid of the Chines Remainder Theorem $R\simeq R_1 \times  R_2$, where $2\in Nil(R_1)$ and $3\in Nil(R_2)$. But, since $R$ is local, we have either $R/J(R) \cong \mathbb{Z}_2$ or $R/J(R) \cong \mathbb{Z}_3$. So, ${\rm M}_2(R/J(R))$ is a strongly NUS-nil clean ring. However, we know that ${\rm M}_2(R/J(R)) = {\rm M}_2(R)/{\rm M}_2(J(R))$, and $J(R)$ is a nilpotent ideal of $R$, because $R$ is Noetherian. Consequently, one knows that ${\rm M}_2(J(R))$ is a nilpotent ideal of ${\rm M}_2(R)$. Finally, Proposition \ref{2.13} leads to the fact that ${\rm M}_2(R)$ is a strongly NUS-nil clean ring, as asked.
\end{proof}

The next assertion also sounds somewhat surprisingly.

\begin{lemma}\label{2.49}
Let $R$ be a strongly NUS-nil clean ring with $2 \in U(R)$ and, for any $u \in U(R)$, we have $u^2 = 1$. Then, $R$ is a commutative ring.
\end{lemma}

\begin{proof}
For all $u, v \in U(R)$, we have $u^2 = v^2 = (uv)^2 = 1$. Therefore, $uv = (uv)^{-1} = v^{-1}u^{-1} = vu$. Hence, the invertible elements commute with each other.\\	
Now, we will illustrate that $R$ is abelian. In fact, for each idempotent $e$ in $R$ and $a \in R$, we know $2e-1\in U(R)$ and $(1+ea(1-e)) \in U(R)$. Since, by what we have shown above, the invertible elements commute with each other, we have $2ea(1-e) =0$. Since $2 \in U(R)$, we get $ea(1-e)=0$, which means $ea = eae$.\\
On the other hand, since $2e-1\in U(R)$ and $(1+(1-e)ae) \in U(R)$, we obtain $2(1-e)ae =0$. So, $ae = eae$. Therefore, $R$ is abelian, as claimed.\\	
On the other side, since $1 + \text{Nil}(R) \subseteq U(R)$ and all invertible elements commute with each other, the nilpotent elements also commute with each other.\\	
On the other hand, since $2 \in U(R)$ and, for every $u \in U(R)$, $u^2 = 1$, we receive $2^2=1$. It now follows that $3\in Nil(R)$. Next, we demonstrate that $a^3-a\in Nil(R)$ for any $a\in R$. To this purpose, choose $a\in R$. If $a\in U(R)$, then $a^2=1$ and so $a^3-a=0$.\\
If, however, $a\not\in U(R)$, then Theorem \ref{7} informs us that $a^4-a^2\in Nil(R)$. So, $a^3-a\in Nil(R)$.\\
Now, given $x, y \in R$. Consulting with \cite[Proposition 2.8]{zhou}, one writes that $x=e_1+e_2+m$ and $y=f_1+f_2+n$, where $e_1,e_2,f_1,f_2\in Id(R)$ and $m,n\in Nil(R)$. As all idempotents are known by the established above to be central in $R$, and nilpotent elements are commuting with each other, we finally conclude that $xy=yx$. Thus, $R$ is a commutative ring, indeed.
\end{proof}

Let $A$, $B$ be two rings and let $M$, $N$ be $(A,B)$-bi-module and $(B,A)$-bi-module, respectively. Moreover, we consider the bilinear maps $\phi :M\otimes_{B}N\rightarrow A$ and $\psi:N\otimes_{A}M\rightarrow B$ that apply to the following properties:
$$Id_{M}\otimes_{B}\psi =\phi \otimes_{A}Id_{M},Id_{N}\otimes_{A}\phi =\psi \otimes_{B}Id_{N}.$$
For $m\in M$ and $n\in N$, define $mn:=\phi (m\otimes n)$ and $nm:=\psi (n\otimes m)$. Now, the $4$-tuple $R=\begin{pmatrix}
	A & M\\
	N & B
\end{pmatrix}$ becomes to an associative ring with obvious matrix operations that is called a {\it Morita context ring}. Designate the two-sided ideals $Im \phi$ and $Im \psi$ to $MN$ and $NM$, respectively, that are called the {\it trace ideals} of the Morita context.

\medskip

We now have at our disposal all the instruments necessary to prove the following criterion.

\begin{proposition}\label{2.41}
Let $R=\left(\begin{array}{ll}A & M \\ N & B\end{array}\right)$ be a Morita context ring such that $MN$ and $NM$ are nilpotent ideals of $A$ and $B$, respectively. Then, $R$ is a strongly NUS-nil clean ring if, and only if, both $A$ and $B$ are strongly square-nil clean rings.
\end{proposition}

\begin{proof}
Since, $MN$ and $NM$ are nilpotent ideals of $A$ and $B$, respectively, one can says that $MN \subseteq J(A)$ and $NM\subseteq J(B)$. Therefore, addapting \cite{9}, we argue that $$J(R)=\begin{pmatrix}
		J(A) & M \\
		N & J(B)
\end{pmatrix}$$ and hence the isomorphism $$\dfrac{R}{J(R)}\cong \dfrac{A}{J(A)}\times \dfrac{B}{J(B)}$$ is fulfilled. Notice that $R$ is a strongly NUS-nil clean ring if, and only if, so is the factor-ring $R/J(R)$. 

Furthermore, thanking to Lemma \ref{2.4}, the quotient $R/J(R)$ is strongly NUS-nil clean ring if, and only if, $\dfrac{A}{J(A)}$ and $\dfrac{B}{J(B)}$ are both strongly square-nil clean rings.\\
Now, if $R$ is strongly NUS-nil clean, then by what we have established so far $J(R)$ is a nil-ideal and so $J(A)$ and $J(B)$ are nil as well. Observe also that Proposition \ref{2.13} can be applied we get that both $A$ and $B$ are strongly square-nil clean rings, as stated.\\ 

Conversely, if $A$ and $B$ are strongly square-nil clean rings, then both $J(A)$ and $J(B)$ are nil, and besides $\dfrac{A}{J(A)}$ and $\dfrac{B}{J(B)}$ are strongly square-nil clean rings yielding that $R/J(R)$ is a strongly NUS-nil clean ring. But, as $J(A)$ and $J(B)$ are nil, we elementarily deduce that $J(R)$ is a nil-ideal of $R$. It now follows that $R$ is a strongly NUS-nil clean ring, as formulated.
\end{proof}

Next, let $R$, $S$ be two rings, and let $M$ be an $(R,S)$-bi-module such that the operation $(rm)s = r(ms$) is valid for all $r \in R$, $m \in M$ and $s \in S$. Given such a bi-module $M$, we consider the {\it formal triangular matrix ring}

$$
{\rm T}(R, S, M) =
\begin{pmatrix}
	R& M \\
	0& S
\end{pmatrix}
=
\left\{
\begin{pmatrix}
	r& m \\
	0& s
\end{pmatrix}
: r \in R, m \in M, s \in S
\right\},
$$
where it obviously forms a ring with the usual matrix operations. Regarding Proposition \ref{2.41}, if we set $N =\{0\}$, then we will obtain the following immediate consequence.

\begin{corollary}\label{2.42}
Let $R,S$ be rings and let $M$ be an $(R,S)$-bi-module. Then, ${\rm T}(R,S,M)$ is strongly NUS-nil clean ring if, and only if, both $R$, $S$ are strongly square-nil clean rings.
\end{corollary}

We now deal with group ring extensions of the NUS-nil clean property as follows: let $R$ be a ring, $G$ a group, and we traditionally denote by $RG$ the group ring of $G$ over $R$. 

In this section, we intend to establish a suitable criterion for a group ring to be strongly NUS-nil clean under some sensible circumstances on the former group and ring objects. Concretely, we are able to achieve this, provided the whole group is locally finite. Recall that a group $G$ is a $p$-group if every element of $G$ has order which is a power of the prime number $p$. Also, we recollect that a group is locally finite if each its finitely generated subgroup is finite.

Suppose now that $G$ is an arbitrary group and $R$ is an arbitrary ring. Standardly, $RG$ the homomorphism $\varepsilon :RG\rightarrow R$ defined by $\varepsilon (\displaystyle\sum_{g\in G}a_{g}g)=\displaystyle\sum_{g\in G}a_{g}$ is called the {\it augmentation map} of $RG$ and its kernel, denoted by $\Delta (RG)$, is called the {\it augmentation ideal} of $RG$.

\medskip

We now proceed by showing the following assertions.

\begin{lemma} \label{3.1}
If $RG$ is a strongly NUS-nil clean ring ring, then so is $R$.
\end{lemma}

\begin{proof}
We know that $RG/\Delta(RG) \cong R$. Thus, it follows at once that $R$ must be a strongly NUS-nil clean ring as being an epimorphic image of $RG$.
\end{proof}

\begin{proposition}\label{3.2}
Let $R$ be a strongly NUS-nil clean ring with $p \in {\rm Nil}(R)$, and let $G$ be a locally finite $p$-group, where $p$ is a prime. Then, the group ring $RG$ is a strongly NUS-nil clean ring.
\end{proposition}

\begin{proof}
Knowing \cite[Proposition 16]{25}, we see that $\Delta(RG)$ is a nil-ideal. Thus, since $RG/\Delta(RG) \cong R$, referring to Proposition \ref{2.13} we infer that $RG$ is a strongly NUS-nil clean ring.
\end{proof}

We also record the following interesting fact.

\begin{lemma}\label{3.7} \cite[Lemma $2$]{27}.
Let $p$ be a prime number with $p\in J(R)$. If $G$ is a locally finite $p$-group, then $\Delta(RG) \subseteq J(RG)$.
\end{lemma}

Having in hand the preceding three statements, we now have at our disposal all the machinery needed to establish the following major assertion.

\begin{theorem}\label{3.8}
Let $R$ be a ring and let $G$ be a locally finite $p$-group, where $p$ is a prime number with $p\in J(R)$. Then, $RG$ is a strongly NUS-nil clean ring if, and only if, $R$ is a strongly NUS-nil clean ring and $\Delta(RG)$ is a nil-ideal of $RG$.
\end{theorem}

We finish off our research work with the following claim.

\begin{lemma}
Let $R$ be a strongly NUS-nil clean ring, and let $G$ be a group such that $\Delta(RG) \subseteq J(RG)$. Then, $RG/J(RG)$ is a strongly NUS-nil clean ring.
\end{lemma}

\begin{proof}
Seeing that $RG = \Delta(RG) + R$, because $\Delta(RG) \subseteq J(RG)$, we write $RG = J(RG) + R$. Therefore, the isomorphism
$$R/(J(RG) \cap R) \cong (J(RG) + R)/J(RG) = RG/J(RG)$$
is true, and since $R$ is strongly NUS-nil clean, one finds that $R/(J(RG) \cap R)$ is strongly NUS-nil clean too. Consequently, we conclude that $RG/J(RG)$ is a strongly NUS-nil clean ring, as requested.
\end{proof}



\end{document}